\documentclass{amsart}
\usepackage{amsfonts}

\newtheorem{theorem}{Theorem}
\theoremstyle{plain}

\newtheorem{definition}{Definition}

\newtheorem{lemma}{Lemma}

\newtheorem{remark}{Remark}

\numberwithin{equation}{section}

\begin{document}
\title[Existence and data dependence results for the fixed points of
multivalued mappings]{Existence and data dependence results for the fixed
points of multivalued mappings}
\author{Emirhan Hac\i o\u{g}lu}
\address{Department of Mathematics, Trakya University, Edirne 22030, Turkey}
\email{emirhanhacioglu@hotmail.com}
\author{Faik G\"{u}rsoy}
\address{Department of Mathematics, Ad\i yaman University, Ad\i yaman 02040,
Turkey}
\email{faikgursoy02@hotmail.com}
\subjclass[2000]{Primary 05C38, 15A15; Secondary 05A15, 15A18}
\keywords{Keyword one, keyword two, keyword three}

\begin{abstract}
We introduce multivalued G\'{o}rnicki mapping and various new types of
multivalued enriched contractive mappings, viz., multivalued enriched Kannan
mapping, multivalued enriched Chatterjea mapping, and multivalued enriched
Ciri\'{c}-Reich-Rus mapping. Then we establish some existence results for
the fixed points of these multivalued contractive type mappings using the
fixed point property of the average operator of the mappings. Furthermore,
we present some data dependence results for the fixed points of mentioned
classes of mappings.
\end{abstract}

\maketitle

\section{Introduction and Preliminaries}

Let $(X,d)$ be a metric space. A single valued mapping $T:X\rightarrow X$ is
said to be

(i) a contraction (see \cite{banach}) if there exist a $\theta \in \lbrack
0,1)$ such that 
\begin{equation*}
d(Tx,Ty)\leq \theta d(x,y)\text{ for all }x,y\in X\text{,}
\end{equation*}%
(ii) a Kannan mapping (see \cite{kannan}) if there exist a $\theta \in
\lbrack 0,\frac{1}{2})$ such that 
\begin{equation*}
d(Tx,Ty)\leq \theta \lbrack d(x,Tx)+d(y,Ty)]\text{ for all }x,y\in X\text{,}
\end{equation*}%
(iii) a Chatterjea mapping mapping (see \cite{chat} )see if there exist a $%
\theta \in \lbrack 0,\frac{1}{2})$ such that 
\begin{equation*}
d(Tx,Ty)\leq \theta \lbrack d(x,Ty)+d(y,Tx)]\text{ for all }x,y\in X\text{,}
\end{equation*}%
(iv) a Ciri\'{c}-Reich-Rus mapping (see \cite{VB9}) if there exist $a,b\in
\lbrack 0,1)$ with $a+2b<1$ such that

\begin{equation*}
d(Tx,Ty)\leq ad(x,y)+b[d(x,Ty)+d(y,Tx)]\text{ for all }x,y\in X\text{.}
\end{equation*}%
A point $z\in X$ with $z=Tz$ is called a fixed point of $T$ and set of fixed
points of $T$ is denoted by $F(T)$.

Fixed point theory studies mainly focus on examining the existence and
approximation of fixed points of various contractive type mappings. The main
purpose of such studies is to determine under which conditions a contractive
type mapping will have fixed points. The renowned Banach contraction
principle is the most basic known result in this context. As the improvement
of the Banach contraction principle in some sense, numerous works have been
devoted to defining various contractive type mappings and studying the
existence of their fixed points as well as their related qualitative
properties. Very recently, again for such a purpose, Berinde \cite{VB1} has
initiated the investigation of the existence of fixed points of
non-contraction mappings by using the fixed point property of the average
operator of the mappings in Banach spaces. The basic idea behind the work of
Berinde is that even though a mapping is not a contraction, the mapping
defined by the convex combination of the value of the mapping at a point and
the point itself can be a contraction. Recent studies in this direction show
that this idea is very effective in studying the existence of fixed points
of mappings that do not belong to one of the known classes of contractive
type mappings such as contraction mapping, Kannan mapping, Chatterjea
mapping, and Ciri\'{c}-Reich-Rus mapping (see \cite%
{VB2,VB3,VB4,VB5,VB6,VB7,VB8,VB9,VB10,Gornicki2,Suantai,Lakshmi}).

The development in the fixed-point theory of multi-valued mappings is
usually depending on the development in the theory of the single-valued
mappings cases. In 1922, Banach \cite{banach} proved the existence of the
fixed point of contraction mappings and suggested a successful approximation
method to the fixed point of contraction mappings. In 1969, Nadler \cite%
{nadler} generalized Banach's result to multivalued contraction mappings and
the same applied to various contractive type mappings in different settings
by numerous researchers.

A mapping $T:X\rightarrow P(X)$ is called multivalued mapping, that is, $T$
maps every point of $X$ to a subset of $X$. It is an indisputable fact that
metric or metric-like structures are needed on the range sets of multivalued
mappings to investigate the existence of fixed points of the multivalued
mappings. The Hausdorff metric $H$ on $CB(X)$, a nonempty, closed and convex
subsets of $X$, is defined by

\begin{equation*}
H(A,B)=\max \{\sup_{x\in A}d(x,B),\sup_{x\in B}d(x,A)\}
\end{equation*}%
\ in which $d(x,B)=\inf \{d(x,y);y\in B\}$ and the distance $\delta $ on $%
CB(X)$ is defined by

\begin{equation*}
\delta (A,B)=\sup \{d(a,b):a\in A,b\in B\}.
\end{equation*}%
A point $z$ is called fixed point of multivalued mapping $T$ if $z\in Tz$
and the set of all fixed points of $T$ is denoted by $F_{T}$. A multivalued
mapping $T:X\rightarrow CB\left( X\right) $ is said to be

(v) a multivalued contraction mapping (see \cite{nadler}) if there exists a $%
\theta \in \lbrack 0,1)$ such that 
\begin{equation*}
H(Tx,Ty)\leq \theta d(x,y)\text{ for all }x,y\in X\text{,}
\end{equation*}%
(vi) a multivalued Kannan mapping (see \cite{damjan}) if there exists a $%
\theta \in \lbrack 0,\frac{1}{2})$ such that 
\begin{equation*}
H(Tx,Ty)\leq \theta \lbrack d(x,Tx)+d(y,Ty)]\text{ for all }x,y\in X\text{,}
\end{equation*}%
(vii) a multivalued Chatterjea mapping (see \cite{Neam}) if there exists a $%
\theta \in \lbrack 0,\frac{1}{2})$ such that 
\begin{equation*}
H(Tx,Ty)\leq \theta \lbrack d(x,Ty)+d(y,Tx)]\text{ for all }x,y\in X\text{,}
\end{equation*}%
(viii) a multivalued Ciri\'{c}-Reich-Rus mapping (see \cite{ciric}) if there
exist $a,b\in \lbrack 0,1)$ with $a+2b<1$ such that

\begin{equation*}
H(Tx,Ty)\leq ad(x,y)+b[d(x,Ty)+d(y,Tx)]\ \text{for all }x,y\in X\text{.}
\end{equation*}

We recall definitions of enriched version of some mappings as follows.

\begin{definition}
Let $\left( X,\left\Vert \cdot ,\cdot \right\Vert \right) $ be a normed
linear space in which $d\left( x,y\right) =\left\Vert x-y\right\Vert $ for
every $x,y\in X$. A mapping $T:X\rightarrow X$ is\ called

\begin{enumerate}
\item an enriched contraction (see \cite{VB1})  if there exist constants $%
b\in \lbrack 0,\infty )$ and $\theta \in \lbrack 0,b+1)$ such that%
\begin{equation*}
d(bx+Tx,by+Ty)\leq \theta d(x,y)\text{, for every }x,y\in X\text{,}
\end{equation*}

\item an enriched Kannan mapping (see \cite{VB2}) if there exist $b\in
\lbrack 0,\infty )$ and $\theta \in \lbrack 0,1/2)$ such that%
\begin{equation*}
d(bx+Tx,by+Ty))\leq \theta \lbrack d(x,Tx)+d(y,Ty)]\text{, for every }x,y\in
X\text{,}
\end{equation*}

\item an enriched Chatterjea mapping (see \cite{VB11}) if there exist $b\in
\lbrack 0,\infty )$ and $\theta \in \lbrack 0,1/2)$ such that%
\begin{equation*}
d(bx+Tx,by+Ty)\leq \theta \lbrack d(x,Ty)+d(y,Tx)]\text{, for every }x,y\in X%
\text{,}
\end{equation*}

\item an enriched Ciri\'{c}-Reich-Rus mapping (see \cite{VB9}) if there
exist $a,b\in \lbrack 0,\infty )$ satisfying $a+2b<1$ such that%
\begin{equation*}
d(bx+Tx,by+Ty)\leq ad(x,y)+b[d(x,Ty)+d(y,Tx)]\text{, for every }x,y\in X%
\text{.}
\end{equation*}
\end{enumerate}
\end{definition}

Popescu \cite{popescu} defined a new contractive type mapping called G\'{o}%
rnicki mapping and provide additional information about enriched
contractions, enriched Kannan mappings and enriched Chatterjea mappings. He
studied the existence of fixed point of G\'{o}rnicki mapping and showed that
a G\'{o}rnicki mapping properly includes enriched contractions, enriched
Kannan mappings and enriched Chatterjea mappings.

\begin{definition}
Let $(X,d)$ be a metric space. A mapping $T:X\rightarrow X$ is\ called a G%
\'{o}rnicki mapping if there exists $M\in \lbrack 0,1)$ such that%
\begin{equation*}
d(Tx,Ty)\leq M[d(x,Tx)+d(y,Ty)+d(x,y)]\text{, for every }x,y\in X\text{,}
\end{equation*}%
and there exist non negative real constants $a$,$b$ with $a<1$, $b\in
\lbrack 0,\infty )$ such that for arbitrary $x\in X$ there exists $u\in X$
with $d(u,Tu)\leq ad(x,Tx)$ and\ $d(u,x)\leq bd(x,Tx)$.
\end{definition}

Abbas et al. \cite{VB10} extended the definition of enriched contraction
mapping to the case of multivalued mapping as follows.

\begin{definition}
Let $(X,\left\Vert \cdot ,\cdot \right\Vert )$ be a normed linear space with 
$d(x,y)=\left\Vert x-y\right\Vert $ for every $x,y\in X$. A mapping $%
T:X\rightarrow CB(X)$ is\ called a multivalued enriched contraction if there
exist constants $b\in \lbrack 0,\infty )$ and $\theta \in \lbrack 0,b+1)$
such that%
\begin{equation*}
H(bx+Tx,by+Ty)\leq \theta d(x,y)\text{, for every }x,y\in X\text{.}
\end{equation*}
\end{definition}

\begin{lemma}
\label{L1}(See \cite{VB10}) Let $(X,\left\Vert \cdot ,\cdot \right\Vert )$
be a linear normed space, $T:X\rightarrow CB(X)$ be a multivalued mapping
and $T_{\lambda }x=\{(1-\lambda )x+\lambda u:u\in Tx\}.$ Then for any $%
\lambda \in \lbrack 0,1)$

$Fix(T)=Fix(T_{\lambda })$.
\end{lemma}

\begin{lemma}
\label{kamran}(See \cite{kamran}) Let $(X,d)$ be a metric space, $A\in CB(X)$%
. Then for any $x\in X,$ and $\mu >1$, there exist an element $a\in A$
satisfying%
\begin{equation*}
d(x,a)\leq \mu d(x,A).
\end{equation*}
\end{lemma}

We define the multivalued G\'{o}rnicki mapping as follows.

\begin{definition}
Let $(X,d)$ be a metric space. A mapping $T:X\rightarrow CB(X)$ is\ called a
multivalued G\'{o}rnicki mapping if there exists $M\in \lbrack 0,1)$ such
that%
\begin{equation*}
H(Tx,Ty)\leq M[d(x,Tx)+d(y,Ty)+d(x,y)]\text{, for every }x,y\in X\text{,}
\end{equation*}%
and there exist non negative real constants $a$,$b$ with $a<1$, $b\in
\lbrack 0,\infty )$ such that for arbitrary $x\in X$ there exists $u\in X$
with $d(u,Tu)\leq ad(x,Tx)$ and\ $d(u,x)\leq bd(x,Tx)$.
\end{definition}

We are now able to generalize, in turn, single-valued enriched Kannan
mappings, enriched Chatterjea mappings, and enriched Ciri\'{c}-Reich-Rus
mappings to the corresponding multivalued mappings as under.

\begin{definition}
Let $(X,\left\Vert \cdot ,\cdot \right\Vert )$ be a normed linear space with 
$d(x,y)=\left\Vert x-y\right\Vert $ for every $x,y\in X$. A mapping $%
T:X\rightarrow CB(X)$ is\ called

\begin{enumerate}
\item a multivalued enriched Kannan mapping if there exist $b\in \lbrack
0,\infty )$ and $\theta \in \lbrack 0,1/2)$ such that%
\begin{equation*}
H(bx+Tx,by+Ty)\leq \theta \lbrack d(x,Tx)+d(y,Ty)]\text{, for every }x,y\in X%
\text{,}
\end{equation*}

\item a multivalued enriched Chatterjea mapping if there exist $b\in \lbrack
0,\infty )$ and $\theta \in \lbrack 0,1/2)$ such that%
\begin{equation*}
H(bx+Tx,by+Ty)\leq \theta \lbrack d((b+1)x,by+Ty)+d((b+1)y,bx+Tx)]\text{,
for every }x,y\in X\text{,}
\end{equation*}

\item a multivalued enriched Ciri\'{c}-Reich-Rus mapping if there exist $%
a,b,c\in \lbrack 0,\infty )$ satisfying $a+2b<1$ such that%
\begin{equation*}
H(cx+Tx,cy+Ty)\leq ad(x,y)+b[d(x,Tx)+d(y,Ty)]\text{, for every }x,y\in X%
\text{.}
\end{equation*}
\end{enumerate}
\end{definition}

\section{Existence results for Multivalued G\'{o}rnicki Mappings}

In this section we prove existence theorems for Multivalued G\'{o}rnicki
Mappings.

\begin{theorem}
\label{gornickI} Let $(X,d)$ be a complete metric space and let $%
T:X\rightarrow CB(X)$ be a multivalued G\'{o}rnicki mapping. Then, $T$ has a
fixed point.

\begin{proof}
Let $x_{0}\in X$ and $d(x_{n+1}Tx_{n+1})\leq ad(x_{n},Tx_{n})$ and $%
d(x_{n+1}x_{n})\leq bd(x_{n},Tx_{n})$. Then we have 
\begin{eqnarray*}
d(x_{n+1}x_{n}) &\leq &bd(x_{n},Tx_{n}) \\
&\leq &bad(x_{n-1},Tx_{n-1}) \\
&\leq &ba^{2}d(x_{n-2},Tx_{n-2}) \\
&&... \\
&\leq &ba^{n}d(x_{0},Tx_{0}),
\end{eqnarray*}

which implies that\ $\{x_{n}\}$ is a Cauchy sequence and $\lim_{n\rightarrow
\infty }$ $d(x_{n},Tx_{n})=0$. Therefore, there exists a $z\in X$ such that 
$\lim_{n\rightarrow \infty }$ $x_{n}=z$. Since%
\begin{eqnarray*}
d(x_{n},Tz) &\leq &d(x_{n},Tx_{n})+H(Tx_{n},Tz) \\
&\leq &d(x_{n},Tx_{n})+M[d(x_{n},Tx_{n})+d(z,Tz)+d(x_{n},z)]\text{,}
\end{eqnarray*}%
\begin{equation*}
d(z,Tz)\leq Md(z,Tz)
\end{equation*}%
which means that $z\in Tz$.
\end{proof}
\end{theorem}

\begin{remark}
If we change definition of multivalued G\'{o}rnicki mapping by one of the
following

\begin{enumerate}
\item[(i)] if there exists $M\in \lbrack 0,1)$ such that%
\begin{equation*}
\delta (Tx,Ty)\leq M[d(x,Tx)+d(y,Ty)+d(x,y)]\text{, for every }x,y\in X\text{%
,}
\end{equation*}%
and there exist $b\in \lbrack 0,\infty )$ and $a\in \lbrack 0,1)$ such that
there exists $u\in X$ with $d(u,Tu)\leq ad(x,Tx)$ and\ $d(u,x)\leq bd(x,Tx)$
for any $x\in X$,

or

\item[(ii)] if there exists $M\in \lbrack 0,1)$ such that 
\begin{equation*}
\delta (Tx,Ty)\leq M[\delta (x,Tx)+\delta (y,Ty)+d(x,y)]\text{, for every }%
x,y\in X\text{,}
\end{equation*}%
and there exist $b\in \lbrack 0,\infty )$ and $a\in \lbrack 0,1)$ such that
there exists $u\in X$ with $\delta (u,Tu)\leq a\delta (x,Tx)$ and\ $%
d(u,x)\leq b\delta (x,Tx)$ for any $x\in X$,
\end{enumerate}

then Theorem \ref{gornickI} still holds. Moreover, $T$ has unique end point,
that is, there exist a point $z$ such that $Tz=\{z\}$.
\end{remark}

\begin{theorem}
\label{gornickII}Let $(X,d)$ be a complete metric space and let $%
T:X\rightarrow CB(X)$ be a multivalued mapping satisfying (i). If $T$ has an
\ a.f.p.s then, $T$ has a unique end point.
\end{theorem}

\begin{proof}
Let $\{x_{n}\}\in X$ be an a.f.p.s of $T$. Then, we have%
\begin{eqnarray*}
\delta (Tx_{n},Tx_{n+m}) &\leq
&M[d(x_{n},Tx_{n})+d(x_{n+m},Tx_{n+m})+d(x_{n},x_{n+m})] \\
&\leq &M[d(x_{n},Tx_{n})+d(x_{n+m},Tx_{n+m})+d(x_{n},Tx_{n}) \\
&&+\delta (Tx_{n},Tx_{n+m})+d(x_{n+m},Tx_{n+m})]\text{,}
\end{eqnarray*}%
which implies%
\begin{equation*}
\lim_{n\rightarrow \infty }\delta (Tx_{n},Tx_{n+m})=0\text{.}
\end{equation*}%
Since%
\begin{equation*}
|d(x_{n},x_{n+m})-\delta (Tx_{n},Tx_{n+m})|\leq
d(x_{n},Tx_{n})+d(x_{n+m},Tx_{n+m})\text{,}
\end{equation*}%
we have $\lim_{n\rightarrow \infty }d(x_{n},x_{n+m})=0$ which means that\ $%
\{x_{n}\}$ is a Cauchy sequence. Therefore, there exists a $z\in X$ such
that $\lim_{n\rightarrow \infty }$ $x_{n}=z$. Then%
\begin{eqnarray*}
\delta (z,Tz) &\leq &d(z,Tx_{n})+\delta (Tx_{n},Tz) \\
&\leq &d(z,Tx_{n})+M[d(x_{n},Tx_{n})+d(z,Tz)+d(x_{n},z)]
\end{eqnarray*}%
implies that 
\begin{equation*}
\delta (z,Tz)\leq Md(z,Tz)\text{.}
\end{equation*}%
Hence, $Tz=\{z\}$. Assume that $T$ has two fixed points $z_{1},z_{2}$. Then,
we have 
\begin{eqnarray*}
d(z_{1},z_{2}) &\leq &\delta (Tz_{1},Tz_{2}) \\
&\leq &M[d(z_{1},Tz_{1})+d(z_{2},Tz_{2})+d(z_{1},z_{2})] \\
&=&Md(z_{1},z_{2})
\end{eqnarray*}%
which means that $z_{1}=z_{2}$.
\end{proof}

\begin{theorem}
Let $(X,d)$ be a complete metric space and let $T:X\rightarrow CB(X)$ be a
multivalued mapping satisfying (ii). If $\ $there exist a sequence $%
\{x_{n}\}\in X$ with $\lim_{n\rightarrow \infty }\delta (x_{n},Tx_{n})=0$,
then $T$ has a unique end point.

\begin{proof}
The proof is similar to the proof of\ Theorem \ref{gornickII}.
\end{proof}
\end{theorem}

\section{Existence results for Multivalued Enriched Mappings}

In this section, we shall prove the existence theorems for the fixed points
of multivalued enriched mappings.

\begin{theorem}
\label{theoenkannan} Let $(X,\left\Vert \cdot ,\cdot \right\Vert )$ be a
normed linear space and let $T:X\rightarrow CB(X)$ be a multivalued enriched
Kannan mapping. Then, $T$ has a fixed point.

\begin{proof}
The mapping $T_{\lambda }$ in Lemma \ref{L1} is a multivalued Kannan
mapping. Indeed, if we let $\lambda =\frac{1}{b+1}$, then we have%
\begin{equation*}
d(x,T_{\lambda }x)=\inf \{d(x,(1-\lambda )x+\lambda u):u\in Tx\}=(1-\lambda
)\inf \{d(x,u):u\in Tx\}=\lambda d(x,Tx),
\end{equation*}%
\begin{eqnarray*}
H(T_{\lambda }x,T_{\lambda }y) &=&\lambda H(Tx+bx,Ty+by) \\
&\leq &\lambda \theta \lbrack d(x,Tx)+d(y,Ty)] \\
&=&\frac{\lambda \theta }{\lambda }[d(x,T_{\lambda }x)+d(y,T_{\lambda }y)] \\
&=&\theta \lbrack d(x,T_{\lambda }x)+d(y,T_{\lambda }y)].
\end{eqnarray*}

Therefore, by Theorem 2.1 in \cite{damjan}, we can find a sequence $%
x_{n+1}\in T_{\lambda }x_{n}=\{(1-\lambda )x_{n}+u_{n}:u_{n}\in Tx_{n}\}$
which is convergent to a fixed point of $T_{\lambda }$.
\end{proof}
\end{theorem}

\begin{theorem}
\label{theocater} Let $(X,\left\Vert \cdot ,\cdot \right\Vert )$ be a normed
linear space and let $T:X\rightarrow CB(X)$ be a multivalued enriched
Chatterjea mapping. Then, $T$ has a fixed point.

\begin{proof}
The mapping $T_{\lambda }$ in Lemma \ref{L1} is a multivalued Chatterjea
mapping. Indeed, if we let $\lambda =\frac{1}{b+1}$, then we have 
\begin{eqnarray*}
H(T_{\lambda }x,T_{\lambda }y) &=&\lambda H(Tx+bx,Ty+by) \\
&\leq &\lambda \theta \lbrack d((b+1)x,by+Ty)+d((b+1)y,bx+Tx] \\
&=&\lambda \theta (b+1)[d(x,\frac{b}{b+1}y+\frac{1}{b+1}Ty)+d(y,\frac{b}{b+1}%
x+\frac{1}{b+1}Tx] \\
&=&\theta \lbrack d(x,T_{\lambda }y)+d(y,T_{\lambda }x]\text{.}
\end{eqnarray*}

Therefore, by Theorem 15 in \cite{Neam}, we can find a sequence $x_{n+1}\in
T_{\lambda }x_{n}=\{(1-\lambda )x_{n}+u_{n}:u_{n}\in Tx_{n}\}$ which is
convergent to a fixed point of $T_{\lambda }$.
\end{proof}
\end{theorem}

\begin{theorem}
\label{theorus} Let $(X,\left\Vert \cdot ,\cdot \right\Vert )$ be a linear
normed space and let $T:X\rightarrow CB(X)$ be a multivalued enriched Ciri%
\'{c}-Reich-Rus mapping. Then, $T$ has a fixed point.

\begin{proof}
The mapping $T_{\lambda }$ in Lemma \ref{L1} is a multivalued Ciri\'{c}%
-Reich-Rus mapping. Indeed, if we let $\lambda =\frac{1}{c+1}$, then we have 
\begin{eqnarray*}
H(T_{\lambda }x,T_{\lambda }y) &=&\lambda H(Tx+cx,Ty+cy) \\
&\leq &\lambda \lbrack ad(x,y)+b[d(x,Tx)+d(y,Ty)]] \\
&\leq &\lambda ad(x,y)+\frac{\lambda b}{\lambda }[d(x,T_{\lambda
}x)+d(y,T_{\lambda }y)] \\
&\leq &ad(x,y)+b[d(x,T_{\lambda }x)+d(y,T_{\lambda }y)].
\end{eqnarray*}

Set $\delta =\frac{a+b}{1-b}$. Then, we get 
\begin{eqnarray*}
H(T_{\lambda }x,T_{\lambda }y) &\leq &ad(x,y)+b[d(x,T_{\lambda
}x)+d(y,T_{\lambda }y)]. \\
&\leq &ad(x,y)+b[d(x,y)+d(y,T_{\lambda }x)+d(y,T_{\lambda }x)+H(T_{\lambda
}x,T_{\lambda }y)]. \\
&\leq &(a+b)d(x,y)+2bd(y,T_{\lambda }x)+bH(T_{\lambda }x,T_{\lambda }y),
\end{eqnarray*}%
which implies

\begin{equation*}
H(T_{\lambda }x,T_{\lambda }y)\leq \delta d(x,y)+2\delta d(y,T_{\lambda }x)%
\text{.}
\end{equation*}%
Similary, we obtain 
\begin{equation*}
H(T_{\lambda }x,T_{\lambda }y)\leq \delta d(x,y)+2\delta d(x,T_{\lambda }y),
\end{equation*}%
and 
\begin{equation*}
H(T_{\lambda }x,T_{\lambda }y)\leq \delta d(x,y)+2\delta d(x,T_{\lambda }x).
\end{equation*}

By Theorem 9 in \cite{maryam}, we can find a sequence $x_{n+1}\in T_{\lambda
}x_{n}=\{(1-\lambda )x_{n}+u_{n}:u_{n}\in Tx_{n}\}$ which is convergent to a
fixed point of $T_{\lambda }$.
\end{proof}
\end{theorem}

\begin{remark}
The single-valued cases of  Theorems \ref{theoenkannan},\ref{theocater},\ref%
{theorus} can be proved easily similar to the proof of Theorems \ref%
{theoenkannan},\ref{theocater},\ref{theorus}. In fact, for the single-valued
cases, it is enough to show that $T_{\lambda }=(1-\lambda )I+\lambda T$ is
satisfied the conditions of the non-enriched version of $T$. Then the proof
will be a repetition of the proof of the non-enriched version of $T$.
\end{remark}

\section{Data dependence results for multivalued mappings}

\begin{theorem}
\label{theokannanstable}Let $(X,\left\Vert \cdot ,\cdot \right\Vert )$ be a
normed linear space with $d(x,y)=\left\Vert x-y\right\Vert $ for every $%
x,y\in X$ and $T:X\rightarrow CB(X)$ be a multivalued enriched Kannan
mapping. If $\ S:X\rightarrow CB(X)$ with $F(S)\neq \emptyset $ and 
\begin{equation*}
H(Tx,Sx)\leq \gamma ,\forall x\in X
\end{equation*}%
$\ $ is satisfied. Then, we can find a $z\in F(T)$ such that \ 
\begin{equation*}
d(z,z^{\ast })\leq \frac{1}{1-k}\lambda \gamma \text{, that is, }%
\sup_{z^{\ast }\in F(S)}d(z^{\ast },F(T))\leq \frac{1}{1-k}\lambda \gamma 
\text{,}
\end{equation*}%
where $k=\frac{\theta }{1-\theta }$ and $z^{\ast }\in F(S)$ is an arbitrary
fixed point. If $S$ is a multivalued enriched Kannan mapping, then we have 
\begin{equation*}
H(F(S),F(T))\leq \frac{1}{1-k}\lambda \gamma .
\end{equation*}

\begin{proof}
Since $T_{\lambda }$ is multivalued Kannan mapping,%
\begin{equation*}
d(u,T_{\lambda }u)\leq kd(x,T_{\lambda }x)\leq kd(x,u)\text{ if }u\in
T_{\lambda }x\text{,}
\end{equation*}%
where $k=\frac{\theta }{1-\theta }<1$. Let $\mu >1$ such that $k\mu <1.$ Let 
$x_{0}=z^{\ast }\in F(S)$ and choose $x_{n+1}\in T_{\lambda }x_{n}$ with $%
d(x_{n},T_{\lambda }x_{n})\mu \geq d(x_{n},x_{n+1})$ for all $n\geq 0$ by
Lemma \ref{kamran}. Then, we have 
\begin{eqnarray*}
d(x_{n+1},x_{n+2}) &\leq &\mu d(x_{n+1},T_{\lambda }x_{n+1}) \\
&\leq &k\mu d(x_{n},T_{\lambda }x_{n}) \\
&\leq &k\mu d(x_{n},x_{n+1}) \\
&\leq &(k\mu )^{2}d(x_{n-1},x_{n}) \\
&&.. \\
&\leq &(k\mu )^{n+1}d(x_{0},x_{1}).
\end{eqnarray*}%
which yields that $\{x_{n}\}$ is a Cauchy sequence and $\lim_{n\rightarrow
\infty }$ $d(x_{n},T_{\lambda }x_{n})=\lim_{n\rightarrow \infty
}d(x_{n},x_{n+1})=0$. Therefore, there exists a $z\in X$ such that $%
\lim_{n\rightarrow \infty }$ $x_{n}=z$. Since 
\begin{equation*}
d(x_{n+1},T_{\lambda }z)\leq H(T_{\lambda }x_{n},T_{\lambda }z)\leq \theta
\lbrack d(x_{n},T_{\lambda }x_{n})+d(z,T_{\lambda }z)]\text{,}
\end{equation*}%
$z\in F(T)$. Now, since $x_{1}=(1-\lambda )x_{0}+\lambda u_{0}\in T_{\lambda
}x_{0}$ in which $u_{0}\in Tx_{0}$, we have 
\begin{eqnarray*}
d(x_{0},x_{1}) &\leq &\mu d(x_{0},T_{\lambda }x_{0}) \\
&=&\mu \inf_{u_{0}\in Tx_{0}}d(x_{0},(1-\lambda )x_{0}+\lambda u_{0}) \\
&=&\mu \inf_{u_{0}\in Tx_{0}}\lambda d(x_{0},u_{0}) \\
&=&\mu \lambda d(x_{0},Tx_{0}) \\
&\leq &\mu \lambda H(Sx_{0},Tx_{0}) \\
&\leq &\mu \lambda \gamma \text{,}
\end{eqnarray*}%
\begin{eqnarray*}
d(x_{n},x_{0}) &\leq &\sum\limits_{m=0}^{n-1}d(x_{m},x_{m+1}) \\
&\leq &\sum\limits_{m=0}^{n-1}(k\mu )^{m}d(x_{1},x_{0}) \\
&\leq &\sum\limits_{m=0}^{n-1}(k\mu )^{m}\mu \lambda \gamma  \\
&=&\frac{1}{1-k\mu }\mu \lambda \gamma \text{.}
\end{eqnarray*}%
If we take limit as $n\rightarrow \infty $ on both sides of the above
inequality, then we obtain 
\begin{equation*}
d(z,x_{0})\leq \frac{1}{1-k\mu }\mu \lambda \gamma ,
\end{equation*}%
which implies that%
\begin{equation*}
\sup_{z^{\ast }\in F(S)}d(F(T),z^{\ast })\leq \frac{1}{1-k}\lambda \gamma 
\end{equation*}%
as $\mu \rightarrow 1^{+}$.

If $S$ is a multivalued enriched Kannan mapping, then, for an arbitrary $%
z\in F(T)$, we can find a $z^{\ast }\in F(S)$ satisfying%
\begin{equation*}
\sup_{z\in F(T)}d(F(S),z)\leq \frac{1}{1-k}\lambda \gamma .
\end{equation*}

Hence,%
\begin{equation*}
H(F(S),F(T))\leq \frac{1}{1-k}\lambda \gamma .
\end{equation*}
\end{proof}
\end{theorem}

\begin{theorem}
Let $(X,\left\Vert \cdot ,\cdot \right\Vert )$ be a normed linear space with 
$d(x,y)=\left\Vert x-y\right\Vert $ for every $x,y\in X$ and $T:X\rightarrow
CB(X)$ be a multivalued enriched Chatterjea mapping. If $\ S:X\rightarrow
CB(X)$ with $F(S)\neq \emptyset $ and 
\begin{equation*}
H(Tx,Sx)\leq \gamma ,\forall x\in X
\end{equation*}%
$\ $ is satisfied. Then, we can find a $z\in F(T)$ such that \ 
\begin{equation*}
d(z,z^{\ast })\leq \frac{1}{1-k}\lambda \gamma \text{, that is, }%
\sup_{z^{\ast }\in F(S)}d(z^{\ast },F(T))\leq \frac{1}{1-k}\lambda \gamma ,
\end{equation*}%
where $k=\frac{\theta }{1-\theta }$ and $z^{\ast }\in F(S)$ \ is an
arbitrary fixed point. If $S$ is a multivalued enriched Chatterjea mapping,
then we have%
\begin{equation*}
H(F(S),F(T))\leq \frac{1}{1-k}\lambda \gamma \text{.}
\end{equation*}

\begin{proof}
Since $T_{\lambda }$ is multivalued Chatterjea mapping,%
\begin{equation*}
d(u,T_{\lambda }u)\leq \frac{\theta }{1-\theta }d(x,T_{\lambda }x)\leq \frac{%
\theta }{1-\theta }d(x,u)\text{ if }u\in T_{\lambda }x\text{.}
\end{equation*}

The rest of \ the proof can be done as in the proof of Theorem \ref%
{theokannanstable}.
\end{proof}
\end{theorem}

\begin{theorem}
Let $(X,\left\Vert \cdot ,\cdot \right\Vert )$ be a normed linear space with 
$d(x,y)=\left\Vert x-y\right\Vert $ for every $x,y\in X$ and $T:X\rightarrow
CB(X)$ be a multivalued enriched Ciri\'{c}-Reich-Rus mapping. If $\
S:X\rightarrow CB(X)$ with $F(S)\neq \emptyset $ and 
\begin{equation*}
H(Tx,Sx)\leq \gamma ,\forall x\in X
\end{equation*}%
$\ $ is satisfied. Then, we can find a $z\in F(T)$ such that \ 
\begin{equation*}
d(z,z^{\ast })\leq \frac{1}{1-k}\lambda \gamma ,\text{ that is, }%
\sup_{z^{\ast }\in F(S)}d(z^{\ast },F(T))\leq \frac{1}{1-k}\lambda \gamma ,
\end{equation*}%
where $k=\frac{\theta }{1-\theta }$ and $z^{\ast }\in F(S)$ \ is an
arbitrary fixed point. If $S$ is a multivalued enriched Ciri\'{c}-Reich-Rus
mapping, then we have 
\begin{equation*}
H(F(S),F(T))\leq \frac{1}{1-k}\lambda \gamma .
\end{equation*}

\begin{proof}
Since $T_{\lambda }$ is multivalued Ciri\'{c}-Reich-Rus mapping, 
\begin{equation*}
H(T_{\lambda }x,T_{\lambda }y)\leq \delta d(x,y)+2\delta d(y,T_{\lambda }x)
\end{equation*}%
for\ $\delta =\frac{a+b}{1-b}<1$. If $u\in T_{\lambda }x$, then we have%
\begin{equation*}
d(u,T_{\lambda }u)\leq H(T_{\lambda }x,T_{\lambda }u)\leq \ \delta
d(x,u)+2\delta d(u,T_{\lambda }x)=\ \delta d(x,u)\text{.}
\end{equation*}

The rest of \ the proof can be done as in the proof of Theorem \ref%
{theokannanstable}.
\end{proof}
\end{theorem}

\begin{theorem}
Let $(X,\left\Vert \cdot ,\cdot \right\Vert )$ be a normed linear space with 
$d(x,y)=\left\Vert x-y\right\Vert $ for every $x,y\in X$ and $T:X\rightarrow
CB(X)$ be a multivalued G\'{o}rnicki mapping with $M<\frac{1}{3}$. If $\
S:X\rightarrow CB(X)$ with $F(S)\neq \emptyset $ and 
\begin{equation*}
H(Tx,Sx)\leq \gamma ,\forall x\in X
\end{equation*}%
$\ $ is satisfied. Then, we can find a $z\in F(T)$ such that \ 
\begin{equation*}
d(z,z^{\ast })\leq \frac{1}{1-k}\lambda \gamma ,\text{ that is, }%
\sup_{z^{\ast }\in F(S)}d(z^{\ast },F(T))\leq \frac{1}{1-k}\lambda \gamma ,
\end{equation*}%
where $k=\frac{\theta }{1-\theta }$ and $z^{\ast }\in F(S)$ \ is an
arbitrary fixed point. If $S$ is a multivalued G\'{o}rnicki mapping, then we
obtain 
\begin{equation*}
H(F(S),F(T))\leq \frac{1}{1-k}\lambda \gamma .
\end{equation*}

\begin{proof}
Since $T$ is multivalued G\'{o}rnicki mapping,%
\begin{eqnarray*}
d(u,Tu) &\leq &H(Tx,Tu) \\
&\leq &M[d(x,Tx)+d(u,Tu)+d(x,u)] \\
&\leq &M[d(x,u)+d(u,Tu)+d(x,u)]\text{ if }u\in Tx\text{,}
\end{eqnarray*}%
which implies that%
\begin{equation*}
d(u,Tu)\leq \frac{2M}{1-M}d(x,u)\text{.}
\end{equation*}

The rest of \ the proof can be done as in\ the proof of Theorem \ref%
{theokannanstable}.
\end{proof}
\end{theorem}

\end{document}